\documentclass[amssymb,12pt]{article}

\usepackage{amssymb}
\usepackage{amsthm}
\usepackage{indentfirst}
\usepackage{amsmath}

\newtheorem{theoremalph}{Theorem}

\newtheorem*{Theorem A}{Theorem A}

\newtheorem{Conjecture}{Conjecture}

\newtheorem{Proposition}{Proposition}[section]
\newtheorem{Lemma}[Proposition]{Lemma}

\newtheorem*{Remark}{Remark}
\newtheorem{Corollary}{Corollary}[Proposition]

 \def\NN{{\mathbb N}}

\def\dim{\operatorname{dim}}

\begin{document}

\title{Dominated chain recurrent class with singularities}
\author{
Christian Bonatti \and Shaobo Gan \and Dawei Yang\footnote{D. Yang thanks the support of NSFC 11001101 and
Ministry of Education of P. R. China 20100061120098. S. Gan is supported by 973 program 2011CB808002 and NSFC
11025101.}}

\date{}

\maketitle


\begin{abstract}
We prove that for generic three-dimensional vector fields, domination implies singular hyperbolicity.

\end{abstract}

\section{Introduction}

One goal of differential dynamical systems is to understand the global dynamics of most dynamical systems. The
dynamics of hyperbolic vector fields are well understood. There is no good description of robustly
non-hyperbolic dynamics. A sequence of conjectures of Palis \cite{Pal00,Pal05,Pal08} gives us a beautiful
perspective. There are already many results in this direction, mainly for diffeomorphisms. If one considers
smooth vector fields, \cite{Pal08} has the following conjectures:

\begin{Conjecture}\label{Con:palis-anydimension}
Every robustly non-hyperbolic vector field can be $C^r$ approximated by a vector field with a homoclinic
tangency or a heterodimensional cycle or a singular cycle.

\end{Conjecture}

\begin{Conjecture}\label{Con:palis-three}
Every robustly non-hyperbolic three-dimensional vector field can be $C^r$ approximated by a vector field with a
singular cycle or a Lorenz-like attractor or a Lorenz-like repeller.

\end{Conjecture}

The difficulty for vector field is the fact that hyperbolic singularities can be accumulated by recurrent
regular points. Usually in this case, one can get a homoclinic orbit of the singularity by perturbation (using
$C^1$ connecting lemmas). Then one can see a different phenomena compared with homoclinic orbits of periodic
orbits: a homoclinic orbit of a singularity cannot be transverse, thus it cannot be robust under perturbation.
The famous Lorenz attractor \cite{Lor63} gives us a very strange phenomenon: it's a robustly non-hyperbolic
chaotic attractor. To understand the hyperbolic properties of Lorenz-like attractor, \cite{MPP98,MPP04} gave the
definition of \emph{singular hyperbolic}. Moreover, for three-dimensional flows, \cite{MPP98,MoP03,MPP04}
proved:
\begin{itemize}

\item Every robustly transitive set is a singular hyperbolic attractor or repeller.

\item If $X$ cannot be accumulated by infinitely many sinks or sources, then $X$ is \emph{singular Axiom A}.

\end{itemize}

In the spirit of results for diffeomorphisms (see \cite{PuS00,BGW07,Cro11b,Cro11p,CrP11}) and vector fields, one
can have several conjectures. To be more precise, we give some definitions. Let $M$ be a compact $C^\infty$
Riemannian manifold without boundary. Let ${\cal X}^r(M)$ be the Banach space of the $C^r$ vector fields on $M$
with the usual $C^r$ norm. For $X\in{\cal X}^1(M)$, $\phi_t^X$ is the flow generated by $X$.
$\Phi_t^X=d\phi_t^X$ is the tangent flow of $X$. If there is no confusion, we will use $\phi_t$ and $\Phi_t$ for
simplicity.  For a compact invariant set $\Lambda$ of $\phi_t$, one says that $\Lambda$ has a dominated
splitting with respect to $\Phi_t$ if there is a continuous invariant splitting $T_\Lambda M=E\oplus F$, and two
constants $C>0$ and $\lambda>0$ such that for any $x\in\Lambda$ and $t>0$, one has
$$\|\Phi_t|_{E(x)}\|\|\Phi_{-t}|_{F(\phi_t(x))}\|\le C e^{-\lambda t}.$$

If $\dim E$ is a constant, then $\dim E$ is called the \emph{index} of the dominated splitting. An invariant
bundle $E\subset T_\Lambda M$ is called \emph{contracting} if there are two constants $C>0$ and $\lambda>0$ such
that for any $x\in\Lambda$ and $t>0$, one has $\|\Phi_t|_{E(x)}\|\le C e^{-\lambda t}$. An invariant bundle $F$
is called \emph{expanding} if it's contracting for $-X$. If $T_\Lambda M=E\oplus F$ is a dominated splitting and
either $E$ is contracting or $F$ is expanding, then one says that $\Lambda$ is \emph{partially hyperbolic}. We
also need the notions of singular hyperbolicity. A continuous invariant bundle $E\subset T_\Lambda M$ is called
\emph{sectional contracting} if there are two constants $C>0$ and $\lambda>0$ such that for any $x\in\Lambda$,
$t>0$ and any two-dimensional subspace $L\subset E(x)$, one has $|{\rm Det} \Phi_t|_{L}|\le C e^{-\lambda t}$. A
continuous invariant bundle $F$ is called \emph{sectional expanding} if it is sectional contracting for $-X$. A
compact invariant set $\Lambda$ is called \emph{singular hyperbolic} if $\Lambda$ has a partially hyperbolic
splitting $T_\Lambda M=E\oplus F$, and either $E$ is sectional contracting and $F$ is sectional expanding, or
$E$ is contracting and $F$ is sectional expanding. Singular hyperbolicity is a generalization of hyperbolicity
for compact invariant set with singularities, it's special for vector fields with singularities. For instance,
the classical Lorenz attractor is singular hyperbolic, but not hyperbolic (see \cite{ArP10}). Here a compact
invariant set $\Lambda$ is called \emph{hyperbolic} if there is a continuous invariant splitting $T_\Lambda
M=E^s\oplus<X>\oplus E^u$, where $E^s$ is uniformly contracting, $E^u$ is uniformly expanding, and $<X>$ is the
subspace generated by the vector field. If $\dim E^s$ is constant, then ${\rm dim}E^s$ is called the
\emph{index} (or stable index) of the hyperbolic set $\Lambda$.

We will consider compact invariant set with recurrence. The most general setting is the chain recurrence. For
any point $x,y\in M$, if for any $\varepsilon>0$, there is a sequence of points $\{x_i\}_{i=0}^n$ and a sequence
of times $\{t_i\}_{i=0}^n$ such that
\begin{itemize}

\item $x_0=x$ and $x_n=y$,

\item $t_i\ge 1$ for each $i$,

\item $d(\phi_{t_i}(x_i),x_{i+1})<\varepsilon$ for any $0\le i\le n-1$.

\end{itemize}

\noindent then one says that $x$ is in the \emph{chain stable set} of $y$. If $x$ is in the chain stable set of $y$ and
$y$ is also in the chain stable set of $x$, then one says that $x$ and $y$ are \emph{chain related}. If $x$ is chain
related with itself, then $x$ is called a \emph{chain recurrent point}. Let ${\rm CR}(X)$ be the set of chain
recurrent points of $X$. It's clear that chain related relation is an equivalent relation. By using this
equivalent relation, one can divide ${\rm CR}(X)$ into equivalent classes. Each equivalent class is called a
\emph{chain recurrent class}. A chain recurrent class is called \emph{non-trivial} if it's not reduced to a
periodic orbit or a singularity.

$X$ is called \emph{singular Axiom A without cycle} if $X$ has only finitely many chain recurrent classes, and
each chain recurrent class is singular hyperbolic. Singular Axiom A vector fields is a generalization of
hyperbolic vector fields. One says that $X$ has a \emph{homoclinic tangency} if $X$ has a hyperbolic periodic
orbit $\gamma$, and $W^s(\gamma)$ and $W^u(\gamma)$ have some non-transversal intersection. One says that $X$
has a \emph{heterodimensional cycle} if $X$ has two hyperbolic periodic orbits $\gamma_1$ and $\gamma_2$ with
different indices such that $W^s(\gamma_1)\cap W^u(\gamma_2)\neq\emptyset$ and $W^u(\gamma_1)\cap
W^s(\gamma_2)\neq\emptyset$.

In the spirit of Palis and the previous results, one has the following conjectures:

\begin{Conjecture}\label{Con:weakPalis}

Every vector field can be $C^r$ approximated by a vector field with a horseshoe or by a Morse-Smale vector
field.

\end{Conjecture}

\begin{Conjecture}\label{Con:strongPalis}

Every vector field can be $C^r$ approximated by one of the following three kinds of vector fields:
\begin{itemize}

\item a vector field which is singular Axiom A without cycle,

\item a vector field with a homoclinic tangency,

\item a vector field with a heterodimensional cycle.

\end{itemize}

\end{Conjecture}

These two conjectures could be viewed as the continuations of the conjectures of Palis for flows. If $\dim M=3$,
Conjecture \ref{Con:strongPalis} was given by \cite{MoP03}. One notices that $X$ is called \emph{star} if there
is a neighborhood ${\cal U}$ of $X$ such that every critical element of $Y\in\cal U$ is hyperbolic. \cite{ZGW08}
conjectured that every star vector field is singular Axiom A without cycle. If Conjecture \ref{Con:strongPalis}
is true, one will have that every generic star vector field is singular Axiom A without cycle.

The main result of this work deals with a special case of Conjecture~\ref{Con:strongPalis}: it concerns when one
can get singular hyperbolicity for three-dimensional vector fields. To avoid some pathological phenomena, one
consider \emph{residual set} of ${\cal X}^r(M)$: it contains a countable intersection of dense open subset of
${\cal X}^r(M)$. Since ${\cal X}^r(M)$ is complete, one has every residual set is dense in ${\cal X}^r(M)$. A
residual set of ${\cal X}^r(M)$ is also called a dense $G_\delta$ set.

\begin{theoremalph}\label{Thm:partialtosingular}

Assume that $\dim M=3$. There is a dense $G_\delta$ set ${\cal G}\subset {\cal X}^1(M)$ such that if $X\in{\cal G}$
and $C(\sigma)$ is a non-trivial chain recurrent class of a singularity $\sigma$ with the following
properties:
\begin{itemize}

\item $C(\sigma)$ contains a periodic point $p$,

\item $C(\sigma)$ admits a dominated splitting $T_{C(\sigma)}M=E\oplus F$ with respect to $\Phi_t$,

\end{itemize}
then $C(\sigma)$ is singular hyperbolic. As a corollary, $C(\sigma)$ is an attractor or a repeller depending on
the index of $\sigma$.

\end{theoremalph}

About the condition that the chain recurrent class contains a periodic point, \cite{Bon11} has the following
conjecture:
\begin{Conjecture}\label{Con:periodicinclass}
For generic vector field $X$, if $C(\sigma)$ is a non-trivial chain recurrent class containing a hyperbolic
singularity $\sigma$, then $C(\sigma)$ contains a hyperbolic periodic orbit.

\end{Conjecture}

If we don't assume that $C(\sigma)$ contains a periodic point, we will get a partially hyperbolic splitting.

\begin{theoremalph}\label{Thm:dominatedtopartial}

Assume that $\dim M=3$. There is a dense $G_\delta$ set ${\cal G}\subset {\cal X}^1(M)$ such that if $X\in{\cal G}$
 and $C(\sigma)$ is a non-trivial chain recurrent class of a singularity $\sigma$ with a dominated
splitting $T_{C(\sigma)}M=E\oplus F$ with respect to $\Phi_t$, then $C(\sigma)$ is partially hyperbolic.

\end{theoremalph}

\begin{Remark}
\cite{BoC04} proved that for $C^1$ generic vector field $X$, $X$ has two kinds of chain recurrent classes:
\begin{itemize}

\item Either, a chain recurrent class contains a hyperbolic periodic orbit, then the chain recurrent class is
the \emph{homoclinic class} of the hyperbolic periodic orbit, i.e., the closure of all transverse homoclinic
orbits of the hyperbolic periodic orbit.

\item Or, a chain recurrent class contains no periodic orbits, then it's called an \emph{aperiodic class}.

\end{itemize}

For diffeomorphisms, it's difficult to define the continuations of aperiodic classes. But for vector fields, if
an aperiodic class contains singularities, it's easy to define their continuations. One of the differences
between periodic orbits and singularities is: singularities can't have transverse homoclinic orbits.

\end{Remark}

We have some further remarks on the conjectures:
\begin{itemize}

\item \cite{ArR03} proved Conjecture~\ref{Con:palis-anydimension} for three-dimensional vector fields for $C^1$
topology.

\item \cite{GaY11} proved Conjecture~\ref{Con:weakPalis} for three-dimensional vector fields for $C^1$ topology.

\end{itemize}

In the spirit of conjectures of \cite{Bon11}, we have the following conjectures:
\begin{Conjecture}\label{Con:robustcycle}
Every vector field can be $C^1$ approximated by a vector field which is singular Axiom A without cycle, or by a
vector field with a robustly heterodimensional cycle.

\end{Conjecture}

\begin{Conjecture}\label{Con:tame}
For $C^1$ generic $X\in{\cal X}^1(M)$ which cannot be approximated by vector fields with a homoclinic tangency, $X$
has only finitely many chain recurrent classes.

\end{Conjecture}

If $\dim M=3$, Conjecture~\ref{Con:robustcycle} can be restated as: $C^1$ generic vector field is singular Axiom
A without cycle. Compared with two-dimensional diffeomorphisms, this conjecture may be called \emph{singular
Smale conjecture}. \cite{CrY11} proves that for $C^1$ generic three-dimensional vector field $X$, the
singularities in a same chain recurrent class of $X$ have the same index.

There are also many results for higher dimensional vector fields. We give a partial list of them.
\begin{itemize}

\item \cite{BPV97} gave an example of singular hyperbolic attractors for which the unstable manifolds of
singularities have arbitrarily large dimension.

\item \cite{LGW05,ZGW08,MeM05} proved that under the \emph{star condition}, every robustly transitive set is
singular hyperbolic.

\item \cite{TuS98} constructed an example on robustly wild strange (quasi)-attractor with singularities for
four-dimensional vector fields for higher regularities.

\item \cite{BKR06} showed that there exist robustly chain transitive non-singular-hyperbolic attractors for five-dimensional
vector fields.

\item \cite{BLY08} showed that there exist robustly chain transitive non-singular-hyperbolic attractors with
different indices of singularities for four-dimensional vector fields.

\end{itemize}

\section{Preliminaries}

\subsection{Dominated splittings}

As in the introduction, every vector field $X\in{\cal X}^1(M)$ generates a flow $\phi_t^X$. We identity the
vector field and its flow as the same object. From the flow $\phi_t^X$, one can define its tangent flow
$\Phi_t^X=d\phi_t^X:~TM\to TM$. For every regular point $x\in M\setminus{\rm Sing}(X)$, one can define its
normal space
$${\cal N}_x=\{v\in T_x M:~<v,X(x)>=0\}.$$
Define the normal bundle on regular points as:
$${\cal N}=\bigsqcup_{x\in M\setminus{\rm Sing}(X)}{\cal N}_x.$$

On the normal bundle ${\cal N}$, one can define the linear Poincar\'e flow $\psi_t^X$: for each $v\in {\cal
N}_x$, one can define
$$\psi_t(v)=\Phi_t(v)-\frac{<\Phi_t(v),X(\phi_t(x))>}{|X(\phi_t(x))|^2}X(\phi_t(x)).$$
For an invariant (may not compact) set $\Lambda\subset M\setminus{\rm Sing}(X)$, one says that $\Lambda$ admits
a dominated splitting with respect to the linear Poincar\'e flow if there are constants $C>0$, $\lambda<0$ and
an invariant splitting ${\cal N}_\Lambda=\Delta^s\oplus\Delta^u$ such that for any $x\in\Lambda$, one has
$\|\psi_t|_{\Delta^s(x)}\|\|\psi_{-t}|_{\Delta^u(\phi_t(x))}\|<Ce^{\lambda t}$. $\dim \Delta^s$ is called the
\emph{index} of the dominated splitting.

If $\Lambda$ is a compact invariant set without singularities, the existence of dominated splitting for the
linear Poincar\'e flow is a robust property.

\begin{Lemma}\label{Lem:robustdominated}
For $X\in{\cal X}^1(M)$, if $\Lambda$ is a compact invariant set which is disjoint from singularities, and
admits a dominated splitting with respect to the linear Poincar\'e flow of index $i$, then there is
$\varepsilon>0$ such that for each $Y$ which is $\varepsilon$-$C^1$-close to $X$, for any compact invariant set
$\Lambda_Y$ contained in the $\varepsilon$ neighborhood of $\Lambda$,  $\Lambda_Y$ admits a dominated
splitting with respect to the linear Poincar\'e flow of index $i$.

\end{Lemma}

For dominated splittings of tangent flows, one will always have the robust property for compact invariant sets
with singularity or not.
\begin{Lemma}\label{Lem:tangentdominatedrobust}
For $X\in{\cal X}^1(M)$, if $\Lambda$ is a compact invariant set with a dominated splitting with respect to the
tangent flow of index $i$, then there is $\varepsilon>0$ such that for each $Y$ which is
$\varepsilon$-$C^1$-close to $X$, for any compact invariant set $\Lambda_Y$ contained in the $\varepsilon$
neighborhood of $\Lambda$, $\Lambda_Y$ admits a dominated splitting with respect to the tangent flow of
index $i$.
\end{Lemma}
By the definition of linear Poincar\'e flow, one has the following lemma:
\begin{Lemma}\label{Lem:tangenttonormal}

For $X\in{\cal X}^1(M)$, if $\Lambda$ is a compact invariant set with a dominated splitting $T_\Lambda M=E\oplus
F$ with respect to the tangent flow and $X(x)\in F(x)$ for any $x\in\Lambda$, then ${\cal
N}_{\Lambda\setminus{\rm Sing}(X)}$ admits a dominated splitting of index $\dim E$ with respect to the linear
Poincar\'e flow.

\end{Lemma}

\subsection{Minimally non-hyperbolic set and $C^2$ arguments}
Sometimes we need to discuss non-hyperbolic set. Its non-hyperbolicity will concentrate on some smaller parts,
which are called \emph{minimally non-hyperbolic set} from Liao \cite{Lia81} and Ma\~n\'e \cite{Man85}.  A
compact invariant set $\Lambda$ is called \emph{minimally non-hyperbolic} if $\Lambda$ is not hyperbolic and
every compact invariant proper subset of $\Lambda$ is hyperbolic. From \cite{ArR03,PuS00}, one has the following
two lemmas.
\begin{Lemma}\label{Lem:minimalnonhyperbolictransitive}
Assume that $\dim M=3$, if $\Lambda$ is a minimally non-hyperbolic set of a vector field $X\in{\cal X}^1(M)$
such that
\begin{itemize}

\item $\Lambda\cap{\rm Sing}(X)=\emptyset$,

\item ${\cal N}_\Lambda$ admits a dominated splitting with respect to the linear Poincar\'e flow,

\end{itemize}
then $\Lambda$ is transitive.

\end{Lemma}

$X$ is called \emph{weak-Kupka-Smale} if every periodic orbit or singularity is hyperbolic.

\begin{Lemma}\label{Lem:pujalssambarino}
Assume that $\dim M=3$ and $X$ is a $C^2$ weak-Kupka-Smale , if $\Lambda$ is a \emph{transitive} minimally
non-hyperbolic set of a vector field $X$ such that $\Lambda\cap{\rm Sing}(X)=\emptyset$, then $\Lambda$ is a
normally hyperbolic torus and the dynamics on $\Lambda$ is equivalent to an irrational flow.

\end{Lemma}

\subsection{Chain recurrence}

A compact invariant set $\Lambda$ is called \emph{chain transitive} if for any $\varepsilon>0$, for any
$x,y\in\Lambda$, there are $\{x_i\}_{i=0}^n\subset\Lambda$ and $\{t_i\}_{i=0}^{n-1}\subset[1,\infty)$ such that
$x_0=x$, $x_n=y$ and $d(\phi_{t_i}(x_i),x_{i+1})<\varepsilon$ for each $0\le i\le n-1$. For chain transitive
sets and hyperbolic periodic orbits or singularities, by using $\lambda$-lemma, one has
\begin{Lemma}\label{Lem:nontrivial}
If $\Lambda$ is a non-trivial chain transitive set and $\Lambda$ contains $\gamma$, where $\gamma$ is a
hyperbolic periodic orbit or a hyperbolic singularity, then $\Lambda\cap
W^s(\gamma)\setminus\{\gamma\}\neq\emptyset$ and $\Lambda\cap W^u(\gamma)\setminus\{\gamma\}\neq\emptyset$.

\end{Lemma}

As a corollary of the above folklore lemma, one has:

\begin{Lemma}\label{Lem:indexsingularity} If $\Lambda$ is a non-trivial chain transitive such that
\begin{itemize}
\item $\Lambda$ admits a dominated splitting $T_{\Lambda}M=E\oplus F$ with respect to the tangent flow and $X(x)\in
F(x)$ for any $x\in\Lambda$,

\item $\Lambda$ contains a hyperbolic singularity $\sigma$,

\end{itemize}

Then ${\rm ind}(\sigma)>\dim E$.
\end{Lemma}

\begin{proof}
We will prove this lemma by absurd. If the lemma is not true, one has ${\rm ind}(\sigma)\le \dim E$ for some
hyperbolic singularity $\sigma\in\Lambda$. Since $\Lambda$ has the dominated splitting, one has
$E^s(\sigma)\subset E(\sigma)$. By Lemma~\ref{Lem:nontrivial}, one has there is $x\in
W^s(\sigma)\cap\Lambda\setminus\{\sigma\}$. Thus, $X(\phi_t(x))\subset T_{\phi_t(x)}W^s(\sigma)$ for any $t>0$.
By the assumption one has $X(\phi_t(x))\subset F(\phi_t(x))$. On the other hand, one has
$\lim_{t\to\infty}<X(\phi_t(x))>\subset E^s(\sigma)\subset E(\sigma)$. This fact contradicts to the continuity
of dominated splittings.

\end{proof}

For each compact set $K$ ($K$ may be not invariant), one can define the chain recurrent set in $K$: ${\rm
CR}(X,K)$. We says that $x\in{\rm CR}(X,K)$ if there is a chain transitive set $\Lambda\subset K$ such that
$x\in \Lambda$. ${\rm CR}(X,K)$ has some upper-semi continuity property.

\begin{Lemma}\label{Lem:uppersemicont}
For given $X$ and $K$, if there is a sequence of vector fields $\{X_n\}$ and a sequence of compact sets $K_n$
such that
\begin{itemize}
\item $X_n\to X$ as $n\to\infty$ in the $C^1$ topology,

\item $K_n\to K$ as $n\to\infty$ in the Hausdorff topology,

\end{itemize}

then $\limsup_{n\to\infty}{\rm CR}(X_n,K_n)\subset {\rm CR}(X,K)$.

\end{Lemma}

By the upper-semi continuity property, one has
\begin{Lemma}\label{Lem:robusthyperbolic}
For given $X$ and $K$, if ${\rm CR}(X,K)$ is hyperbolic, then there is a $C^1$ neighborhood ${\cal U}$ of $X$
and a neighborhood $U$ of $K$ such that ${\rm CR}(Y,\overline{U})$ is hyperbolic.

\end{Lemma}

\begin{Lemma}\label{Lem:chaintransitiveempty}

For given $X$ and $K$, if ${\rm CR}(X,K)=\emptyset$, then there is a $C^1$ neighborhood ${\cal U}$ of $X$ and a
neighborhood $U$ of $K$ such that ${\rm CR}(Y,\overline{U})=\emptyset$.

\end{Lemma}
\subsection{Ergodic closing lemma for flows}

Ma\~n\'e's ergodic closing lemma \cite{Man82} was established for flows by Wen \cite{Wen96}. $x\in
M\setminus{\rm Sing}(X)$ is called \emph{strongly closable} if for any $C^1$ neighborhood $\cal U$ of $X$, for
any $\delta>0$, there are $Y\in{\cal U}$ and $p\in M$, $\pi(p)>0$ such that
\begin{itemize}

\item $\phi_{\pi(p)}^Y(p)=p$,

\item $X(x)=Y(x)$ for any $x\in M\setminus\cup_{t\in[0,\pi(p)]}B(\phi_t(x),\delta)$,

\item $d(\phi_t^X(x),\phi_t^Y(p))<\delta$ for each $t\in[0,\pi(p)]$.

\end{itemize}

Let $\Sigma(X)$ be the set of strongly closable points of $X$.
\begin{Lemma}\label{Lem:ergodicclosing}[Ergodic closing lemma for flows \cite{Wen96}]
$\mu(\Sigma(X)\cup{\rm Sing}(X))=1$ for every $T>0$ and every $\phi_T^X$-invariant probability Borel measure $\mu$.

\end{Lemma}

\subsection{Generic results}

We list all known generic results we need in this paper.
\begin{Lemma}\label{Lem:generic}
There is a dense $G_\delta$ set ${\cal G}\subset{\cal X}^1(M)$ such that for each $X\in{\cal G}$, one has
\begin{enumerate}

\item Every periodic orbit or every singularity of $X$ is hyperbolic.

\item For any non-trivial chain recurrent class $C(\sigma)$, where $\sigma$
is a hyperbolic singularity of index $\dim M-1$, then every separatrix of $W^u(\sigma)$ is dense in $C(\sigma)$.
As a corollary, $C(\sigma)$ is transitive.

\item Given $i\in[0,\dim M-1]$. If there is a sequence of vector fields $\{X_n\}$ such that
\begin{itemize}

\item $\lim_{n\to\infty}X_n=X$,

\item each $X_n$ has a hyperbolic periodic orbits $\gamma_{X_n}$ of index $i$ such that $\lim_{n\to\infty}\gamma_{X_n}=\Lambda$,

\end{itemize}
then there is a sequence of  hyperbolic periodic orbits $\gamma_n$ of index $i$ \emph{of $X$} such that
$\lim_{n\to\infty}\gamma_n=\Lambda$.

\end{enumerate}

\end{Lemma}

\begin{Remark}
Item 1 is the classical Kupka-Smale theorem \cite{Kup64,Sma63}. Item 2 is a corollary of the connecting lemma
for pseudo-orbits \cite{BoC04}. There is no explicit version like this. \cite[Section 4]{MoP03} gave some ideas
about the proof of Item 2 without using of the terminology of chain recurrence. Item 3 is fundamental, one can
see \cite{Wen04} for instance.

\end{Remark}

\subsection{The saddle value of a singularity}

Assume that $\dim M=d$. For a hyperbolic singularity $\sigma$ of $X\in{\cal X}^r(M)$, one can list all
eigenvalues of $DX(\sigma)$ as $\{\lambda_1,\lambda_2,\cdots,\lambda_i,\lambda_{i+1},\cdots,\lambda_d\}$ such
that
$${\rm Re}(\lambda_1)\le{\rm Re}(\lambda_2)\le\cdots\le{\rm Re}(\lambda_i)<0<{\rm Re}(\lambda_{i+1})\le\cdots\le{\rm Re}(\lambda_d).$$
Then one says that $I(\sigma)={\rm Re}(\lambda_i)+{\rm Re}(\lambda_{i+1})$ is the \emph{saddle value} of
$\sigma$.

By using the $C^1$ connecting lemma for pseudo-orbits \cite{BoC04} and an estimation of Liao \cite{Lia89},
\cite{GaY11} proved that
\begin{Lemma}\label{Lem:sectionalcontractingisolated}
Assume that $\dim M=3$. There is a residual set ${\cal G}\subset{\cal X}^1(M)$ such that for any $X\in{\cal G}$,
if $\sigma$ is a hyperbolic singularity of index $2$ and $I(\sigma)<0$ and the norms of eigenvalues of
$DX(\sigma)$ are mutually different, then $\sigma$ is isolated in ${\rm CR}(X)$: there is a neighborhood $U$ of
$\sigma$ such that $U\cap {\rm CR}(X)=\{\sigma\}$.

\end{Lemma}

\begin{Remark}
We give some rough idea of the proof of Lemma~\ref{Lem:sectionalcontractingisolated}. Let $\sigma$ be a
singularity as in Lemma~\ref{Lem:sectionalcontractingisolated}. If it's not isolated form other chain recurrent
points (i.e., $C(\sigma)$ is non-trivial), by using the $C^1$ connecting lemma, one can get a homoclinic loop
associate to the singularity. By an extra perturbation, one can assume that the homoclinic loop is normally
hyperbolic. By another small perturbation, one can put the unstable manifold of the singularity in the stable
manifold of a sink: and this is a robust property! Thus, the unstable manifold of the singularity is in the
stable manifold of a sink generically, which gives a contradiction.

\end{Remark}

One notices that \cite{MPP04} proved that singularities with the properties in
Lemma~\ref{Lem:sectionalcontractingisolated} is disjoint from robustly transitive sets for three-dimensional
flows.

\section{Lyapunov chain recurrent classes: Proof of Theorem~\ref{Thm:partialtosingular}}

\begin{Lemma}\label{Lem:dominatedwithoutsingularity}

Assume that $\dim M=3$. For $C^1$ generic $X\in{\cal X}^1(M)$, if $\Lambda$ is a chain transitive set with the following properties:
\begin{itemize}

\item ${\rm Sing}(X)\cap\Lambda=\emptyset$,

\item ${\cal N}_\Lambda=\Delta^s\oplus\Delta^u$ is a dominated splitting with respect to $\psi_t$,

\end{itemize}
then $\Lambda$ is hyperbolic.

\end{Lemma}

\begin{proof}
We take a countable basis $\{U_n\}$ of $M$. Let ${\cal O}=\{O_n\}_{n\in\NN}$ such that each $O_n$ is the union
of finite elements in $\{U_n\}$. For each $n$, one can define
\begin{itemize}

\item ${\cal H}_n\subset{\cal X}^1(M)$ is a subset with the following property: $X\in{\cal H}_n$ if and only if
${\rm CR}(X,\overline{O_n})$ is hyperbolic or ${\rm CR}(X,\overline{O_n})=\emptyset$. By
Lemma~\ref{Lem:robusthyperbolic} and Lemma~\ref{Lem:chaintransitiveempty}, ${\cal H}_n$ is an open set.

\item ${\cal N}_n\subset{\cal X}^1(M)$ is a subset with the following property: $X\in{\cal N}_n$ if and only if
there is a $C^1$ neighborhood ${\cal U}\subset{\cal X}^1(M)$ of $X$ such that for any $Y\in\cal U$, ${\rm
CR}(Y,\overline{O_n})$ is neither hyperbolic nor empty.
\end{itemize}
By definitions, one has ${\cal H}_n\cup{\cal N}_n$ is open and dense in ${\cal X}^1(M)$. Now one takes
$${\cal G}=\bigcap_{n\in\NN}({\cal H}_n\cup{\cal N}_n).$$
It's clear that ${\cal G}$ is a dense $G_\delta$ set. For each $X\in{\cal G}$, we assume that $\Lambda$ is a
non-singular chain transitive set with a dominated splitting ${\cal N}_\Lambda=\Delta^s\oplus\Delta^u$ on the
normal bundle ${\cal N}_\Lambda$ with respect to the linear Poincar\'e flow $\psi_t$. We will prove that
$\Lambda={\rm CR}(X,\Lambda)$ (since $\Lambda$ is chain transitive) is hyperbolic. If not, $\Lambda$ contains a
minimally non-hyperbolic set. By Lemma~\ref{Lem:minimalnonhyperbolictransitive}, $\Lambda$ contains a minimally
non-hyperbolic set $\Gamma$, which is transitive. Take $n\in\NN$ such that

\begin{itemize}
\item $\Gamma$ is contained in $O_n$.

\item There is a $C^1$ neighborhood $\cal U$ of $X$ such that the maximal invariant set in $\overline{O_n}$ of
$Y\in{\cal U}$ has a dominated splitting on the normal bundle with respect to $\psi_t^Y$ by
Lemma~\ref{Lem:robustdominated}.

\end{itemize}

Since $\Gamma$ is not hyperbolic, one has $X\in{\cal N}_n$. Take a $C^2$ weak-Kupka-Smale vector field
$Y\in{\cal N}_n\cap{\cal U}$. Since $Y$ in ${\cal N}_n$, ${\rm CR}(Y,\overline{O_n})$ is not hyperbolic. But
since $Y\in{\cal U}$, the maximal invariant set has a dominated splitting on the normal bundle with respect to
the linear Poincar\'e flow. By Lemma~\ref{Lem:pujalssambarino}, one has the splitting ${\rm
CR}(Y,\overline{O_n})=\Lambda_1\cap\Lambda_2$ such that
\begin{itemize}
\item $\Lambda_1$ is the union of chain transitive sets, and each chain transitive set is hyperbolic. In other
words, $\Lambda_1$ is hyperbolic.

\item $\Lambda_2=\cup_{1\le i\le m}{\mathbb T}^2_i$, each ${\mathbb T}^2_i$ is a normally hyperbolic torus, and
the dynamics on ${\mathbb T}^2_i$ is equivalent to an irrational flow. In other words, ${\mathbb T}^2_i$ is
isolated from other chain recurrent points.

\end{itemize}

By an arbitrarily small perturbation, there is $Z$ $C^1$-close to $Y$ such that
\begin{itemize}

\item ${\rm CR}(Z,\overline{O_n})=\Lambda_1\cup\Lambda_2$.

\item $\Lambda_1$ is still a hyperbolic set of $Z$.

\item The dynamics on ${\mathbb T}^2_i$ of $Z$ is Morse-Smale.

\end{itemize}

As a corollary, ${\rm CR}(Z,\overline{O_n})$ is hyperbolic. This fact contradicts to $Z\in{\cal N}_n$.

\end{proof}

\begin{Lemma}\label{Lem:approximatedbysink}

Assume that $\dim M=3$. For $C^1$ generic $X\in{\cal X}^1(M)$, if $\Lambda$ is a compact invariant set with a
dominated splitting $T_{\Lambda} M=E\oplus F$ of index 1 with respect to the tangent flow $\Phi_t$ such that
\begin{itemize}
\item There is $T>0$ such that for every singularity $\sigma\in\Lambda$, one has $|{\rm
Det}(\Phi_T|_{F(\sigma)})|>1$.

\item For every $x\in\Lambda\setminus{\rm Sing}(X)$, one has $<X(x)>\subset F(x)$.

\item $F$ is not sectional expanding,

\end{itemize}
then there is a sequence of sinks $\{P_n\}$ such that $\lim_{n\to\infty}P_n=\Gamma\subset\Lambda$.

\end{Lemma}

\begin{proof}
One define $\varphi(x)=\log|{\rm Det}(\Phi_T|_{F(x)})|$ for each $x\in\Lambda$. We will prove this lemma by
absurd. If for any $x\in\Lambda$, there is $n(x)\in\NN$ such that $\varphi(\phi_{n(x)T}(x))>0$, then by a
compact argument one can get that $F$ is sectional expanding. Thus, there is an ergodic invariant measure $\mu$
with ${\rm supp}(\mu)\subset\Lambda$ such that
$$\int \varphi d\mu\le 0.$$
By Lemma~\ref{Lem:ergodicclosing}, for the set of strongly closable set $\Sigma(X)$, one has
$\mu(\Sigma(X)\cup{\rm Sing}(X))=1$. Since for each singularity $\sigma\in\Lambda$ one has ${\rm
Det}(\Phi_T|_{F(\sigma)})>1$ by assumption, one gets $\mu(\Sigma(X)\setminus{\rm Sing}(X))=1$. Since $\varphi$
is a continuous function, by Birkhoff's ergodic theorem, one has for almost every point $x\in{\rm
supp}(\mu)\cap\Sigma(X)$ with respect to $\mu$ such that
$$\lim_{n\to\infty}\frac{1}{n}\sum_{i=0}^{n-1}\varphi(\phi_{i T}(x))=\int\varphi d\mu.$$

Without loss of generality, one has  that $x$ is not periodic. Otherwise, since $X$ is $C^1$ generic, one has
that $x$ is a hyperbolic periodic point. This will imply that the orbit of $x$ is a periodic sink in $\Lambda$.
Thus one can get the conclusion.

Since $x$ is a strong closable point, for any $\varepsilon>0$ there are $Y$ which is $\varepsilon$-$C^1$-close
to $X$ and $p_\varepsilon\in M$, $\pi(p_\varepsilon)>0$ such that
\begin{itemize}

\item $\phi_{\pi(p_\varepsilon)}^Y(p_\varepsilon)=p_\varepsilon$,

\item $d(\phi_t^X(x),\phi_t^Y(p_\varepsilon))<\delta$ for each $t\in[0,\pi(p_\varepsilon)]$.

\end{itemize}

Since $x$ is non-periodic, one has $\pi(p_\varepsilon)\to\infty$ as $\varepsilon\to 0$. By the continuity
property of dominated splittings, one has the orbit of $p_\varepsilon$ with respect to $Y$ also a dominated
splitting $E_\varepsilon\oplus F_\varepsilon$ and $F_\varepsilon\to F$, $E_\varepsilon\to E$ as $\varepsilon\to
0$ by the Grassman metric. As a corollary, one has
$$\lim_{\varepsilon\to 0}\frac{1}{[\pi(p_\varepsilon)/T]}\sum_{i=0}^{n-1}\log|{\rm Det}(D\Phi_T^Y|_{F_\varepsilon(\phi_{i T}(p_\varepsilon))})|\le 0.$$

Since one has the dominated splitting on the orbit of each periodic orbit, one has for each $p_\varepsilon$ the
largest Lyapunov exponent along the orbit of $p_\varepsilon$ tends to zero as $\varepsilon\to 0$.

By a Lemma of Franks for flows \cite{Man82,BGV06}, in this case, one can change the index of $\{{\rm
Orb}(p_\varepsilon)\}$ to be smaller by an arbitrarily small perturbation. As a corollary, there is a sequence
of vector fields $\{X_n\}$ such that
\begin{itemize}
\item $\lim_{n\to\infty}X_n=X$ in the $C^1$ topology.

\item Each $X_n$ has a sink $\gamma_n$ such that $\lim_{n\to\infty}\gamma_n=\Gamma$.

\end{itemize}

Since $X$ is $C^1$ generic, by Lemma~\ref{Lem:generic}, one can gets the conclusion.

\end{proof}

Now we will manage to prove Theorem~\ref{Thm:dominatedtopartial}. Assume that {\bf we are under the assumptions
of Theorem~\ref{Thm:dominatedtopartial}.} First we have
\begin{Lemma}\label{Lem:flowdirection}

For every regular point $x\in C(\sigma)$, one has
\begin{itemize}
\item either, $X(x)\in E(x)$,

\item or, $X(x)\in F(x)$.

\end{itemize}
\end{Lemma}
\begin{proof}

By Lemma~\ref{Lem:generic}, $C(\sigma)$ is transitive. Thus, the set
$$\tilde{C}=\{x\in C(\sigma):~\omega(x)=\alpha(x)=C(\sigma)\}$$
is dense in $C(\sigma)$.

By the invariance property, if for some $y\in \tilde{C}$, one has either $X(y)\in E(y)$ or $X(y)\in F(y)$, then
one can get the conclusion. Thus, one can assume that it's not true. Thus, for any $y\in\tilde{C}$, one has
$X(y)\notin E(y)\cup F(y)$. Take $y_0\in\tilde{C}$. There are a sequence of times $\{t_n\}$ such that
\begin{itemize}

\item $\lim_{n\to\infty}t_n=\infty$.

\item $\lim_{n\to\infty}\phi_{t_n}(y_0)=y_0$.

\end{itemize}
By the dominated property, one will have $\lim_{n\to\infty}\Phi_{t_n}(X(y_0))\in F(y_0)$. Since
$\Phi_{t_n}(X(y_0))=X(\phi_{t_n}(y_0))$ and the vector field $X$ is continuous with respect to the space
variable $x\in M$, one has $X(y_0)=\lim_{n\to\infty}X(\phi_{t_n}(y_0))=\lim_{n\to\infty}\Phi_{t_n}(X(y_0))$.
This will imply that $X(y_0)\in F(y_0)$, which gives a contradiction.

\end{proof}

\begin{Corollary}\label{Cor:singular-flowdirection}

If ${\rm ind}(\sigma)=2$, then for any regular point $y\in C(\sigma)$, $X(y)\in F(y)$. As a corollary,
singularities in $C(\sigma)$ will have the same index.

\end{Corollary}

\begin{proof}
The fact the ${\rm ind}(\sigma)=2$ implies that $E(\sigma)\subset E^s(\sigma)$. We will prove this corollary by
absurd. If it's not true, by Lemma~\ref{Lem:flowdirection}, one has $X(x)\in E(x)$ for any regular point $x\in
C(\sigma)$. This implies that regular points will approximate $\sigma$ only in the stable subspace
$E^s(\sigma)$. By $\lambda$-lemma, one knows that regular points will accumulate both $W^s(\sigma)$ and
$W^u(\sigma)$. This fact gives a contradiction.

If singularities in $C(\sigma)$ have different indices, then there are hyperbolic singularities
$\sigma_1,\sigma_2\in C(\sigma)$ such that ${\rm ind}(\sigma_1)=1$ and ${\rm ind}(\sigma_2)=2$. Thus by previous
arguments, one has for every regular point $x$, one has $X(x)\in E(x)$ and $X(x)\in F(x)$. This contradiction
ends the proof.

\end{proof}

\begin{Lemma}\label{Lem:onebundlehyperbolic}
If ${\rm ind}(\sigma)=2$, then $\dim E=1$ and $E$ is contracting.

\end{Lemma}

\begin{proof}
First by Corollary~\ref{Cor:singular-flowdirection}, one has for every regular point $x$, $X(x)\in F(x)$. If
$\dim E=1$ is not true, one has $\dim E=2$. This means that regular points approximate $\sigma$ only in the
unstable subspace of $\sigma$, which contradicts to the fact that $C(\sigma)\cap
W^s(\sigma)\setminus\{\sigma\}\neq\emptyset$.

We will prove that $E$ is contracting. One notices that by Corollary~\ref{Cor:singular-flowdirection}, every
singularity $\sigma'$ in $C(\sigma)$ has index 2 and $E(\sigma')\subset E^s(\sigma')$. For any point
$x\in\Sigma$, there are two cases:
\begin{enumerate}

\item $\omega(x)\subset{\rm Sing}(X)$,

\item $\omega(x)\setminus {\rm Sing}(X)\neq\emptyset$.

\end{enumerate}

In the first case, one has there is $t_x>0$ such that $\|\Phi_{t_x}|_{E(x)}\|<1$. In the second case, one choose
$y\in \omega(x)\setminus{\rm Sing}(X)$. Take a small neighborhood $U_y$ of $y$ such that for any $y_1,y_2\in
U_y$, one has
$$\frac{1}{2}\le \frac{|X(y_1)|}{|X(y_2)|}\le 2.$$
Choose a sequence of times $\{t_n\}$ such that
\begin{itemize}
\item $\lim_{n\to\infty}t_n=\infty$.

\item $\phi_{t_n}(x)\in U_y$.

\end{itemize}

Thus,
$$\frac{|X(x)|}{|X(\phi_{t_n}(x))|}=\frac{|X(x)|}{|X(\phi_{t_1}(x))|}\frac{|X(\phi_{t_1}(x))|}{|X(\phi_{t_n}(x))|}\le 2\frac{|X(x)|}{|X(\phi_{t_1}(x))|}.$$
Since $E\oplus F$ is a dominated splitting and $X\subset F$, one has there are constants $\lambda<0$ and $C>0$
such that
$$\|\Phi_{t_n}|_{E(x)}\|\le C e^{\lambda t_n}\frac{|X(x)|}{|X(\phi_{t_n}(x))|}\le 2C e^{\lambda t_n}\frac{|X(x)|}{|X(\phi_{t_1}(x))|}.$$
When $n$ is large enough, one has $\|\Phi_{t_n}|_{E(x)}\|<1$

By summarizing the above arguments, in any case, one has for any $x\in C(\sigma)$, there is $t_x>0$ such that
$\|\Phi_{t_x}|_{E(x)}\|<1$. By a classical compact argument, one has $E$ is uniformly contracting.

\end{proof}

Since $\dim M=3$, every hyperbolic singularity in a non-trivial chain recurrent class has either index 1 or
index 2, Lemma~\ref{Lem:onebundlehyperbolic} ends the proof of Theorem~\ref{Thm:dominatedtopartial}.

We will manage to prove Theorem~\ref{Thm:partialtosingular} now.

\begin{proof}[\bf Proof of Theorem~\ref{Thm:partialtosingular}]
Without loss of generality, one can assume that every singularity in $C(\sigma)$ has index 2. Thus, $C(\sigma)$
is Lyapunov stable.
\begin{itemize}

\item By Lemma~\ref{Lem:onebundlehyperbolic}, $C(\sigma)$ has a partially hyperbolic splitting
$T_{C(\sigma)}M=E^s\oplus F$ with $\dim E^s=1$. Moreover, every singularity in $C(\sigma)$ has index 2.

\item By Lemma~\ref{Lem:sectionalcontractingisolated}, $I(\sigma')>0$ for any singularity $\sigma'\in
C(\sigma)$.

\end{itemize}

We will prove Theorem~\ref{Thm:partialtosingular} by absurd. If $C(\sigma)$ is not singular hyperbolic, then $F$
is not sectional expanding by Theorem B. By Lemma~\ref{Lem:approximatedbysink}, one has that there is a sequence
of sinks $\{P_n\}$ such that $\lim_{n\to\infty}P_n=\Lambda\subset C(\sigma)$. There are two cases:

\begin{enumerate}

\item $\Lambda$ contains a singularity $\sigma'\in C(\sigma)$.

\item $\Lambda\cap {\rm Sing}(X)=\emptyset$.

\end{enumerate}

In the second case, by Lemma~\ref{Lem:dominatedwithoutsingularity}, $\Lambda$ is a hyperbolic set. Then, either
$\Lambda$ is a sink, or $\Lambda$ cannot be accumulated by sinks. But if $\Lambda$ is a sink, $\Lambda$ cannot
be contained in the chain recurrent class $C(\sigma)$, which shows that the second case is impossible.

Now we are in the first case. Since $\Lambda$ contains a singularity $\sigma'\in C(\sigma)$, one has
$\Lambda\cap W^u(\sigma')\neq\emptyset$. Since $X$ is $C^1$ generic, one has that every separatrix of
$W^u(\sigma')$ is dense in $C(\sigma)$ by Lemma~\ref{Lem:generic}. Thus, $\Lambda=C(\sigma)$. As a corollary,
$\Lambda$ contains the hyperbolic periodic point $p$. Thus, there are $p_n\in P_n$ such that
$\lim_{n\to\infty}p_n=p$. Since $C(\sigma)$ is Lyapunov stable by Lemma~\ref{Lem:generic}, one has
$W^u_{loc}({\rm Orb}(p))\subset C(\sigma)$. It's clear that $W^u_{loc}({\rm Orb}(p))$ is a two-dimensional
manifold and transversal to the strong stable direction $E^s$. In particular, one has that $W^s(W^u_{loc}({\rm
Orb}(p)))$ contains ${\rm Orb}(p)$ as its interior. As a corollary, $\omega(x)\subset C(\sigma)$ for any $x\in
W^s(W^u_{loc}({\rm Orb}(p)))$ and $W^s(W^u_{loc}({\rm Orb}(p)))$ contains no sinks. This contradicts to the fact
that $p$ can be accumulated by sinks.

Now notice that one of the main theorems in \cite{MoP03} asserts that every singular hyperbolic chain recurrent
class with a singularity is an attractor or a repeller for three-dimensional flows.

\end{proof}

\vskip 5pt

\noindent Christian Bonatti

\noindent Institut de Math\'ematiques de Bourgogne, Universit\'e de Bourgogne(Dijon),  Dijon, 21004  FRANCE

\noindent bonatti@u-bourgogne.fr

\vskip 5pt

\noindent Shaobo Gan

\noindent School of Mathematics Sciences, Pekin University, Beijing, 100871, P.R.China

\noindent gansb@math.pku.edu.cn

\vskip 5pt

\noindent Dawei Yang

\noindent School of Mathematics, Jilin University, Changchun, 130012, P.R. China

\noindent yangdw1981@gmail.com

\end{document}